\documentclass{article}

\usepackage[english]{babel}
\usepackage{lmodern}            
\usepackage[T1]{fontenc}         
\usepackage[utf8]{inputenc}      
\usepackage{bbold}
\usepackage{xspace}
\usepackage{biblatex} 
\usepackage{amsmath,amssymb}
\usepackage{amsthm}
\usepackage{newtxmath}
\usepackage{stmaryrd}
\usepackage{dsfont}
\usepackage[margin=3.2cm]{geometry}
\usepackage{hyperref}
\usepackage{xcolor}
\usepackage{enumitem}
\hypersetup{
    colorlinks,
    linkcolor={black},
    citecolor={black},
    urlcolor={black}}
\usepackage{graphicx}
\graphicspath{ {./images/} }
\input{insbox}

\newtheorem*{remark}{Remark}
\newtheorem{theorem}{Theorem}[section]

\newtheorem{lemme}[theorem]{Lemma}
\newtheorem{proposition}[theorem]{Proposition}

\theoremstyle{definition}
\newtheorem{definition}{Definition}[section]

\newcommand{\E}{\mathbb{E}}		

\newcommand{\M}{\mathbb{M}}
\newcommand{\N}{\mathbb{N}}		

\renewcommand{\P}{\mathbb{P}}	
\newcommand{\R}{\mathbb{R}}		
\renewcommand{\S}{\mathbb{S}}	

\newcommand{\cF}{\mathcal F}		
\newcommand{\cG}{\mathcal G}	
\newcommand{\cH}{\mathcal H}	
\newcommand{\cM}{\mathcal M}	
\newcommand{\cN}{\mathcal N}	
\newcommand{\cP}{\mathcal P}		
\newcommand{\cS}{\mathcal S}
\newcommand{\cT}{\mathcal T}		

\newcommand{\cW}{\mathcal W}

\newcommand{\1}{\mathds{1}}			

\usepackage{biblatex} 
\title{There are no geodesic hubs in the Brownian sphere}
\author{Mathieu Mourichoux}
\date{}

\begin{document}\maketitle
\begin{abstract}
    A point of a metric space is called a $k$-hub if it is the endpoint of exactly $k$ disjoint geodesics, and that the concatenation of any two of these paths is still a geodesic. We prove that in the Brownian sphere, there is no $k$-hub for $k\geq 3$. 
\end{abstract}

\section{Introduction}
The Brownian sphere $(\cS,D)$ is a model of random geometry, that arises as the scaling limit of several models of random planar maps. In particular, it is the scaling limit of quadrangulations of the sphere with $n$ faces chosen uniformly at random \cite{uniqueness,convergence}. The Brownian sphere also comes with a volume measure $\mu$. In this work, we are interested in the existence of a family of exceptional points in the Brownian sphere, which are called \textbf{geodesic hubs}. 

Recall that a geodesic $(\gamma(t))_{t\in[0,\tau]}$ in a metric space $(E,d)$ is a path $\gamma:[0,\tau]\mapsto E$ such that, for every $s,t\in[0,\tau],\,d(\gamma(s),\gamma(t))=|t-s|$. We say that a point $x\in E$ is a \textbf{geodesic hub} with at least $k$ arms, or a $k^+$-hub, if :
\begin{itemize}[label=\textbullet]
    \item there exists at least $k$ geodesics $\gamma_i:[0,\tau_i]\mapsto E$ such that $\gamma_i(0)=x$
    \item for every 
$1\leq i<j\leq k$, $\gamma_i((0,\tau_i])\cap\gamma_j((0,\tau_j])=\varnothing$ and the path obtained by following $\gamma_i$ from $\gamma_i(\tau_i)$ to $x$ and then $\gamma_j$ from $x$ to $\gamma_j(\tau_j)$ is a geodesic.
\end{itemize} The geodesics $(\gamma_i)_{1\leq i\leq k}$ are called the arms of the $k^+$-hub, and we say that the $k$-uple $(x_i)_{1\leq i\leq k}$ \textbf{borders a $k^+$-hub}. Of course, a $k^+$-hub is bordered by infinitely many points. We say that $x\in E$ is a \textbf{$k$-hub} if it is a $k^+$-hub, but not a $(k+1)^+$-hub. This notion was introduced in \cite{Poissonroads}, in the course of studying some random fractal metric on $\R^2$. 
\begin{figure}
        \centering
        \includegraphics[scale=0.4]{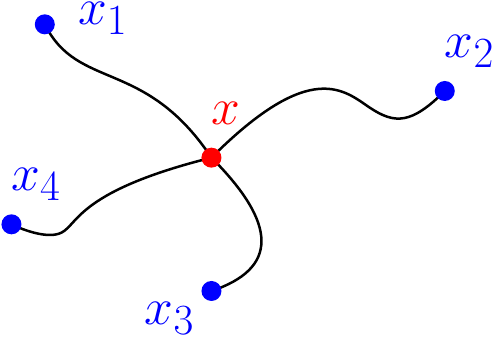}
        \caption{Illustration where $x$ is a $4$-hub, and $(x_1,x_2,x_3,x_4)$ border a $4$-hub. Each of the black paths between $x_i$ and $x_j$ for $i\neq j$ in $\{1,2,3,4\}$ is a geodesic.}
    \end{figure}
Note that there always exists $2$-hubs in a geodesic space, since $2$-hubs are just points in the interior of geodesics. The main contribution of this paper is to prove that $3^+$-hubs do not exist in the Brownian sphere, confirming a prediction of \cite{Poissonroads}.

\begin{theorem}\label{hub}
    Almost surely, there is no $3^+$-hub in the Brownian sphere.
\end{theorem}

Of course, this result implies that there is no $k$-hub for $k\geq 3$ either. Let us mention that exceptional points in the Brownian sphere, in particular geodesic stars, have already been studied in several works. Recall that a point $x\in E$ is a $k^+$-star if there exists $k$ disjoint geodesics emanating from $x$, and we say that $x$ is a $k$-star if it is a $k^+$-star but not a $(k+1)^+$-star. It was proved in \cite{geodesic2,Geodesicstars} that for $1\leq k\leq 4$, the set of $k$-star in the Brownian sphere has Hausdorff dimension $5-k$, almost surely. However, the existence of $5$-stars remains an open question. Similarly, it is not known if two geodesics can intersect each other at a single point which is in the interior of both geodesics.

Finally, the paper \cite{Poissonroads} studies several exceptional points for a random metric on $\R^2$, constructed from a Poisson process of roads. In particular, in stark contrast with the Brownian sphere, they prove that their model contains $k$-hubs up to $k=4$. 

Let us sketch the proof of Theorem \ref{hub} :

\begin{itemize}[label=\textbullet]
    \item First, we prove that we can restrict ourselves to the study of $3^+$-hubs bordered by points distributed according to the volume measure $\mu$, and construct a variant of the Brownian sphere with three marked points. This essentially relies on results about approximations of geodesics from \cite{geodesic2}, and Bismut decomposition of a random labelled tree under $\N_0$.
    \item Then, we prove that the existence of a $3^+$-hub implies that with a positive probability, a geodesic passes through the apex of a Brownian slice (see section \ref{slice} for a definition of this space). To do so, we rely on the characterization of geodesics towards $x_*$ from \cite{geodesic1}, and the Palm formula. 
    \item Then, we show that with a positive probability, there is no geodesic that passes through the apex of a Brownian slice, and we conclude with a $0-1$ argument. This part of the proof mostly relies on explicit formulas for Poisson point measures.
\end{itemize}

This paper is organized as follows. In Section 2, we introduce the notions of Brownian snake and Brownian spheres that will be used in this article. Then, Section 3 is devoted to the proof of Theorem \ref{hub}. Finally, in Section 4, we use Theorem \ref{hub} to prove Proposition \ref{alignement}, which is of independent interest. 

\subsection*{Acknowledgements}
I am grateful to my supervisor Grégory Miermont for his support, and for his careful reading of this paper. I would also like to thank Lou Le Bihan and Simon Renouf for their help in typing this work.

\section{Preliminaries}

In this section, we introduce the notion of Brownian sphere and Brownian slice. To do so, we first recall some basic notions about snake trajectories.

\subsection{Snake trajectories}\label{intervalle}
Here, we recall the definition and some basic notions about snake trajectories. A finite path is a continuous function $w:[0,\zeta]\longrightarrow \R$, where $\zeta=\zeta(w)\geq 0$ is called the lifetime of $w$, and we set $\widehat{w}=w(\zeta)$. We write $\mathfrak{W}$ for the set of all finite paths in $\R$, and for every $x\in\R$, we let $\mathfrak{W}_x:=\{w\in\mathfrak{W}:w(0)=x\}$. The set $\mathfrak{W}$ is a Polish space when equipped with the distance
\begin{equation*}
    d(w,w')=|\zeta(w)-\zeta(w')|+\sup_{t\geq 0}|w(t\wedge\zeta(w))-w'(t\wedge\zeta(w'))|.
\end{equation*}
Finally, we identify the point $x\in\R$ with the element of $\mathfrak{W}_x$ with zero lifetime. 
\begin{definition}
Fix $x\in\R$. A snake trajectory starting from $x\in\R$ is a continuous mapping $s\mapsto\omega_s$ from $\R_+$ to $\mathfrak{W}_x$ which satisfies the following conditions : 
\begin{itemize}
    \item $\omega_0=x$ and the quantity $\sigma(\omega)=\sup\{s\geq 0:\omega_s\neq x\}$ is finite, 
    \item For every $0\leq s\leq s'$, we have $\omega_s(t)=\omega_{s'}(t)$ for every $t\in[0,\min_{s\leq r\leq s'}\zeta(\omega_r)]$. 
\end{itemize}
\end{definition}
The quantity $\sigma(\omega)$ is called the duration of the snake trajectory $\omega$. We will denote by $\mathfrak{S}_x$ the set of snake trajectories starting from $x\in\R$, and $\mathfrak{S}=\bigcup_{x\in\R}\mathfrak{S}_x$ the set of all snake trajectories. We will use the notation $W_s(\omega)=\omega_s$ and $\zeta_s(\omega)=\zeta(\omega_s)$. Note that a snake trajectory $\omega$ is completely determined by its lifetime function $s\mapsto\zeta_s(\omega)$ and its tip function $s\mapsto\widehat{W}_s(\omega)$ (see \cite{Refserpent} for a proof). We also write $W_*(\omega)=\inf_{t\geq0}\widehat{W}_t(\omega)$.\\
Given a snake trajectory $\omega\in\S$, its lifetime function $\zeta(\omega)$ encodes a compact $\R$-tree, which will be denoted by $\cT_{\omega}$. More precisely, if we introduce a pseudo-distance on $[0,\sigma(\omega)]$ by letting 
\begin{equation*}
    d_{(\omega)}(s,s')=\zeta_s(\omega)+\zeta_{s'}(\omega)-2\min_{s\wedge s'\leq r\leq s\vee s'}\zeta_r(\omega),
\end{equation*}
then $\cT_{\omega}$ is the quotient space $[0,\sigma(\omega)]/\{d_{(\omega)}=0\}$ equipped with the distance induced by $d_{(\omega)}$. We write $p_\cT:[0,\sigma(\omega)]\rightarrow \cT_\omega$ for the canonical projection, and root the tree $\cT_{\omega}$ at $\rho_\cT:=p_\cT(0)=p_\cT(\sigma(\omega))$. The tree $\cT_{\omega}$ also comes with a volume measure, which is the pushforward of the Lebesgue measure on $[0,\sigma(\omega)]$ by the projection $p_\cT$. Finally, note that, because of the snake property, $W_s(\omega)=W_{s'}(\omega)$ if $p_\cT(s)=p_\cT(s')$. In particular, the mapping $s\rightarrow\widehat{W}_s(\omega)$ can be viewed as a function on the tree $\cT_\omega$. In this article, for $u\in\cT_\omega$ and $s\in [0,\sigma]$ such that $p_\cT(s)=u$, we will often use the notation $\ell_u=\widehat{W}_s(\omega)$. \\
We also define intervals on the tree $\cT_\omega$ as follows.  For every $s,t\in[0,\sigma]$ with $t<s$, we use the convention that $[s,t]=[s,\sigma]\cup[0,t]$. For every $u,v\in\cT_\omega$, there is a smallest interval $[s,t]$ such that $p_\cT(s)=u$ and $p_\cT(t)=v$, and we define 
\[[u,v]:=\{p_\cT(r):r\in[s,t]\}.\]

\subsection{The Brownian snake excursion measure}\label{snake}

In this subsection, we give the construction and some properties of the Brownian snake (see \cite{serpent} for more details). For every $x\in\R$, we define a $\sigma$-finite measure on $\mathfrak{S}_x$, called the Brownian snake excursion measure and denoted as $\N_x$, as follows. Under $\N_x$ :
\begin{enumerate}
    \item The lifetime function $(\zeta_s)_{s\geq 0}$ is distributed according to the Itô measure of positive excursions of linear Brownian motion, normalized so that the density of $\sigma$ under $\N_x$ is $t\mapsto(2\sqrt{2\pi t^3})^{-1}$.
    \item Conditionally on $(\zeta_s)_{s\geq 0}$, the tip function $(\widehat{W}_s)_{s\geq0}$ is a Gaussian process with mean $x$ and covariance function : 
    \begin{equation*}
        K(s,s')=\min_{s\wedge s'\leq r\leq s\vee s'}\zeta_r.
    \end{equation*}
\end{enumerate}
The measure $\N_x$ is also an excursion measure away from $x$ for the Brownian snake, which is a Markov process in $\mathfrak{W}_x$. For every $t>0$, we can define the conditional probability measure $\N_x^{(t)}=\N_x(\cdot\,|\,\sigma=t)$, which can also be constructed by replacing the Itô measure used to define $\N_x$ by the law of a Brownian excursion with duration $t$.

For every $y<x$, we have
\begin{equation}\label{inf}
    \N_x(W_*<y)=\frac{3}{2(x-y)^2}.
\end{equation}
(see \cite{serpent} for a proof). Therefore, we can define the conditional probability measure $\N_x(\cdot\,|\,W_*<y)$. Moreover, one can prove that under $\N_x$ or $\N_x^{(t)}$, a.e, there exists a unique $s_*\in[0,\sigma]$ such that $\widehat{W}_{s_*}=W_*$ (see e.g. Proposition 2.5 in \cite{Conditionnedbrowniantrees}).

Finally, these measures satisfy a scaling property. For every $\lambda>0$ and $\omega\in\mathfrak{S}_x$, we define $\Theta_\lambda(\omega)\in\mathfrak{S}_{x\sqrt{\lambda}}$ by $\Theta_\lambda(\omega)=\omega'$ with
\begin{equation}\label{scaling}
    \omega'_s(t):=\sqrt{\lambda}\,\omega_{s/\lambda^2}(t/\lambda),\quad\text{for $s\geq 0$ and $0\leq t\leq\zeta_s':=\lambda\zeta_{s\lambda^2}$}. 
\end{equation}
Then, the pushforward of $\N_x$ by $\Theta_\lambda$ is $\lambda\N_{x\sqrt{\lambda}}$, and for every $t>0$, the pushforward of $\N_x^{(t)}$ by $\Theta_\lambda$ is $\N_{x\sqrt{\lambda}}^{(\lambda^2t)}$.

\subsection{The Brownian sphere}\label{sphere}

Fix a snake trajectory $\omega\in\mathfrak{S}_0$ with duration $\sigma$.  We introduce, for every $u,v\in\cT_\omega$,
\[D^\circ_{(\omega)}(u,v)=\ell_u+\ell_v-2\max\bigg(\min_{r\in[u,v]}\ell_r,\min_{r\in[v,u]}\ell_r\bigg)\]
and  
\begin{equation}
    D_{(\omega)}(u,v)=\inf\bigg\{\sum_{i=1}^p D^\circ_{(\omega)}(u_i,u_{i-1})\bigg\}
\end{equation}
where the infimum is taken over all integers $p\geq 1$ and sequences $u_0,...,u_p\in\cT_\omega$ such that $u_0=u$ and $u_p=v$. Note that $D_{(\omega)}\leq D^\circ_{(\omega)}$.\\
Observe that $D^\circ_{(\omega)}(u,v)\geq|\ell_u-\ell_v|$, which translates into a simple (but very useful) bound: 
\begin{equation}\label{Bound}
    D_{(\omega)}(u,v)\geq|\ell_u-\ell_v|.
\end{equation}
The mapping $(u,v)\mapsto D_{(\omega)}(u,v)$ defines a pseudo-distance on $\cT_\omega$. This allows us to introduce a quotient space $\cT_\omega/\{D_{(\omega)}=0\}$, which is equipped with the distance naturally induced by $D_{(\omega)}$. \\
We can now apply the previous construction with a random snake trajectory. 
\begin{definition}
    The standard Brownian sphere is defined under the probability measure $\N_0^{(1)}$ as the random metric space $\cS=\cT/\{D=0\}$ equipped with the distance $D$, and a volume measure $\mu$ which is the pushforward of the volume measure on $\cT$ under the canonical projection $p_\cS:\cT\rightarrow\cS$. 
\end{definition}

Observe that the labelling function $\ell$ can be defined on $\cS$. Therefore, for every $x\in\cS$, we let by $\ell_x$ stands for the label of $x$.

We also introduce the free Brownian sphere, which is defined in the same way replacing $\N_0^{(1)}$ by $\N_0$; even though this is not a random variable anymore, it is often more convenient to work with this object. Note that we can also see the standard Brownian sphere (or the free Brownian sphere) as a quotient of $[0,1]$ (or $[0,\sigma]$). We will sometimes use this point of view, and we write $\mathbf{p}:[0,1]\rightarrow\cS$ for the canonical projection. We also set $x_0=\mathbf{p}(0)$.

As mentioned earlier, almost surely, there exists a unique element $u_*\in\cT$ such that $\ell_{u_*}=\inf_{u\in\cT}\ell_u=W_*$. Therefore, we write $x_*=p_{\cS}(u_*)$ and $\ell_*=\ell_{x_*}$. Note that the bound \eqref{Bound} together with the inequality $D\leq D^\circ$ implies that almost surely, for every $x\in\cS$, 
\[D(x,x_*)=\ell_x-\ell_*.\]
In particular, we have 
\[D(x_0,x_*)=-\ell_*.\]
The following proposition, proved in \cite{TopologicalStructure}, completely characterizes the points of $\cT$ that are identified in the Brownian sphere.
\begin{proposition}\label{Identification}
    Almost surely, for every $u,v\in\cT$, we have 
    \[D(u,v)=0\Longleftrightarrow D^\circ(u,v)=0.\]
\end{proposition}

In Section 4, we will prove that Theorem \ref{hub} implies the following result, which is of independent interest.

\begin{proposition}\label{alignement}
    Let $u,v\in\cT$ such that 
   \[D^\circ(u,v)=D(u,v).\]
   Then, $(x_*,p_{\cS}(u),p_{\cS}(v))$ are aligned, meaning that they are on a common geodesic.  
\end{proposition}

\subsection{Brownian slices}\label{slice}

Here, we introduce the notion of Brownian slice, which plays a major role in the proofs to come. Fix a snake trajectory $\omega\in\mathfrak{S}_0$ with duration $\sigma$. We define a pseudo-distance $\Tilde{d}$ on $[0,\sigma]$ by 
\begin{equation*}
    \Tilde{d}_{(\omega)}(s,t)=\widehat{W}_s+\widehat{W}_t-2\inf_{r\in [s\land t,s\lor t]}\widehat{W}_r.
\end{equation*}
Then, similarly to what we did to construct the Brownian sphere, we can define a pseudo-distance $\Tilde{D}_{(\omega)}^\circ$ on $\cT_\omega$:
\begin{equation}\label{pseudo dist slice}
    \Tilde{D}_{(\omega)}^\circ(u,v)=\inf\{\Tilde{d}_{(\omega)}(s,t):s,t\in[0,\sigma],p_{\cT_\omega}(s)=u,p_{\cT_\omega}(t)=v\}.
\end{equation}
The difference with the distance $D^\circ$ of Section \ref{sphere} is that we forbid  ``to go around the root of $\cT_\omega$''  when computing the distance. 
Finally, we can define another pseudo-distance on $\cT_\omega$ :
\begin{equation*}
    \Tilde{D}_{(\omega)}(u,v)=\inf_{u_0,...,u_p}\sum_{i=1}^p\Tilde{D}_{(\omega)}^\circ(u_i,u_{i-1})
\end{equation*}
where the infimum is taken over every $p\in\N^*$ and sequences in $\cT$ such that $u_0=u$ and $u_p=v$.
\begin{definition}
    The \textit{free Brownian slice} is defined under the measure $\N_0$ as the metric space $\Tilde{\cS}=\cT/\{\Tilde{D}=0\}$, equipped with the distance $\Tilde{D}$. We write $p_{\Tilde{\cS}}:\cT\rightarrow \Tilde{\cS}$ for the canonical projection, $\Tilde{\mathbf{p}}$ for the projection $[0,\sigma]\rightarrow\Tilde{\cS}$ and $\Tilde{\rho}=\Tilde{\mathbf{p}}(0)$.
\end{definition}
This space has already been studied in \cite{uniqueness} to prove the convergence of quadrangulations toward the Brownian sphere, and in \cite{Browniandisk} to prove the convergence of quadrangulation with a boundary toward the Brownian disk (see also \cite{Geodesicstars}). It is also the scaling limit of some models of random planar maps with geodesic boundaries.

Let us explain how this space is related to the Brownian sphere $\cS$. It was proved in \cite{geodesic1} that almost surely, there exists a unique geodesic $\Gamma$ in $\cS$ between $x_0$ and $x_*$. Then, if we cut $\cS$ along the geodesic $\Gamma$, the resulting space is a Brownian slice. In particular, this space has a boundary made of two geodesic segments, which correspond to the geodesic $\Gamma
$ that has been cut (see \cite[Section 3.2]{uniqueness} for more details).

Note that because $d\leq\Tilde{d}$, we have $D\leq\Tilde{D}$. Furthermore, $\Tilde{\cS}$ has the same scaling property as the Brownian sphere. 

\subsection{Coding labelled trees with triples}\label{Triple}

Here, we briefly explain how to encode a labelled tree by a triple $(X,\mathcal{N}_l,\mathcal{N}_r)$. We refer to \cite[Section 2.4]{Spinedecomposition} for more details.

Consider a triple $(X,\mathcal{N}_l,\mathcal{N}_r)$, where 
\begin{itemize}[label=\textbullet]
    \item $X=(X_t)_{t\in[0,h]}$ is a random path,
    \item $\mathcal{N}_l$ and $\mathcal{N}_r$ are two random point measures on $[0,h]\times\mathfrak{S}$.    
\end{itemize}
Then, under some natural assumptions, one can define a labelled tree $\cT$ from this triple, made of a spine of length $h$, and where each atom $(t_i,\omega_i)$ of $\mathcal{N}_l$ (respectively $\mathcal{N}_r$) represents a labelled subtree isometric $\cT_{\omega_i}$ branching off the left side (respectively the right side) of the spine at height $t_i$. Moreover, the labels on the spine are given by the process $X$. This tree is rooted at the bottom of the spine, and has a distinguished point, which is the top of the spine. 

One can also define an exploration process for the tree $\cT$, which allows us to define intervals on this tree. Furthermore, it is possible to represent the labelled tree $\cT$ by a snake trajectory $\omega\in\mathfrak{S}$ such that $\cT_\omega=\cT$. Therefore, one can construct a random metric space $(S,d)$ and a projection $p_S:\cT\rightarrow S$ from any admissible triple $(X,\mathcal{N}_l,\mathcal{N}_r)$, as explained in Section \ref{sphere} and \ref{slice}. 

\section{Proof of Theorem \ref{hub}}

As mentioned in the introduction, the proof will consist of three steps.

\subsection{Marking three points in the Brownian sphere}

We start by giving a construction of a Brownian sphere with a distinguished triple $(x_0,x_1,x_*)$ of typical points, and prove that we can restrict our study to this model. 
First, we show that the set of compact metric spaces with three distinguished points which bordered a $3^+$-hub is Borel. We refer to \cite[Section 6.4]{tessalations} for details about the marked Gromov-Hausdorff topology.
\begin{lemme}\label{mesurable}
    Let $\mathbb{M}^{\bullet\bullet\bullet}$ be the set of isometry classes of triply-pointed compact metric spaces, equipped with the Gromov-Hausdorff topology. Then, the set 
    \[\mathcal{H}=\{(M,x_1,x_2,x_3)\in\mathbb{M}^{\bullet\bullet\bullet},\,M\text{ is a geodesic space and }(x_1,x_2,x_3)\text{ borders a $3^+$-hub }\}\] is a Borel set. 
\end{lemme}
\begin{proof}
First, recall that the set of geodesic spaces is closed in $\M$ (see \cite[Theorem 7.5.1]{burago}). Therefore, in what follows, every space considered is a geodesic space.

    For every $n,m\geq1$, let $\mathcal{H}_{n,m}$ be the set of $(M,x_1,x_2,x_3)\in\mathbb{M}^{\bullet\bullet\bullet}$ such that 
    \begin{itemize}
        \item for every $i\neq j$ in $\{1,2,3\}$,
        ,
        \[d(x_i,x_j)>\frac{1}{m},\]
        \item There exists $w\in M$ such that for every $i\neq j$ in $\{1,2,3\}$,
        \begin{equation}\label{Presque hub}
            d(x_i,w)+d(w,x_j)<d(x_i,x_j)+\frac{1}{n}.
        \end{equation}
    \end{itemize}
    The first condition guarantees that $(x_1,x_2,x_3)$ are disjoint points, and the second one means that they almost border a $3^+$-hub. Observe that $\mathcal{H}_{n,m}$ is an open set. Therefore, the set 
    \[\bigcup_{m\geq1}\bigcap_{n\geq1}\mathcal{H}_{n,m}\] is a Borel set. Moreover, we clearly have 
    \[\mathcal{H}\subset\bigcup_{m\geq1}\bigcap_{n\geq1}\mathcal{H}_{n,m}\] (we can choose $w$ in \eqref{Presque hub} as the $3^+$-hub bordered by $(x_1,x_2,x_3)$). Let us show a converse inclusion, which will give the desired result. Fix $(M,x_1,x_2,x_3)\in\mathbb{M}^{\bullet\bullet\bullet}$ and suppose that there exists $m\geq1$ such that 
    \[M\in\bigcap_{n\geq1}\mathcal{H}_{n,m}.\]
    For every $n\geq1$, consider $w_n\in M$ such that \eqref{Presque hub} holds with this choice. By compactness, we can suppose that the sequence $(w_n)_{n\geq1}$ converges toward some element $w_\infty\in M$. Moreover, for every $i\neq j$ in $\{1,2,3\}$, we have 
        \begin{equation*}
            d(x_i,w_\infty)+d(w_\infty,x_j)=d(x_i,x_j).
        \end{equation*}
Since $x_1,x_2$ and $x_3$ are disjoint elements, we can easily deduce from this equality that $w_\infty$ is a $3^+$-hub, which gives $M\in\mathcal{H}$, and conclude the proof.   
\end{proof}

Now, we show that we just need to consider $3^+$-hubs bordered by typical points. 

\begin{proposition}\label{reduction}
    Let $(x_0,x_1,x_2)$ be three points of the standard Brownian sphere $\cS$ distributed according to the volume measure $\mu$. Then, 
    \[\P(\text{There exists a $3^+$-hub})>0\quad\text{ if and only if }\quad\P((x_0,x_1,x_2)\text{ borders a $3^+$-hub})>0.\]    
\end{proposition}
\begin{proof}
    Since $\left\{\text{There exists a $3^+$-hub}\right\}\supset\{(x_0,x_1,x_2)\text{ borders a $3^+$-hub})\}$, one implication is straightforward. Conversely, suppose that there exists a $3^+$-hub with a positive probability. If a triple $(u,v,w)$ borders such a hub, by a result of approximations of geodesics \cite[Theorem 1.7]{geodesic2}, there exist neighborhoods $(U,V,W)$ of $(u,v,w)$ such that for every $x_0\in U,x_1\in V,x_2\in W$, the triple $(x_0,x_1,x_2)$ borders a $3^+$-hub. This proves the result, since these neighborhoods have a strictly positive $\mu$-measure. 
\end{proof}

Then, we will define the random trees and surfaces that we will be dealing with. For every $a>0$, let $B^{(a)}=(B^{(a)}_t)_{t\in[0,a]}$ be a Brownian motion starting from $0$ of duration $a$, and given $B^{(a)}$, let $\mathcal{N}_l^{(a)}$ and $\mathcal{N}_r^{(a)}$ be two independent Poisson point measures on $[0,a]\times\S$, with intensity 
\[2\1_{[0,a]}(t)\N_{B_t^{(a)}}(d\omega)dt.\]

As explained in Section \ref{Triple}, we can associate a random labelled tree $\cT_a$ to the triple $(B^{(a)},\mathcal{N}_l^{(a)},\mathcal{N}_r^{(a)})$. This tree has two distinguished points, called $\rho_0$ and $\rho_a$, which are respectively the bottom and the top of the spine. Let $\cS_a$ be the random metric space associated to $\cT_a$, and $p_{\cS_a}:\cT_a\rightarrow \cS_a$. This space comes with three distinguished points, which are 
\[x_0=p_{\cS_a}(\rho_0),\quad x_a=p_{\cS_a}(\rho_a),\quad x_*=p_{\cS_a}(u_*).\]
Now, we will explain how these trees and spaces are related to the measure $\N_0$.

Arguing for $\N_0(d\omega)$, for every $s\in(0,\sigma)$, we can encode  the labelled subtrees branching off the ancestral line of $p_\cT(s)$ by two point measures $\cP_l^{(s)}$ and $\cP_r^{(s)}$. More precisely, we consider the connected components $(u_i,v_i),i\in I$ of the open set $\{r\in[0,s]:\zeta_r(\omega)>\min_{t\in[r,s]}\zeta_t(\omega)\}$. For every $i\in I$, we can define a snake trajectory $\omega^i$ of duration $\sigma(\omega^i)=v_i-u_i$, by setting for every $r\in[0,\sigma(\omega^i)]$,
\[\omega^i_r(t)=\omega_{u_i+r}(\zeta_{u_i}(\omega)+t),\quad\text{ for }0\leq t\leq \zeta_{\omega_r^i}=\zeta_{u_i+r}(\omega)-\zeta_{u_i}(\omega).\] Then, we can define a point measure $\cP_l^{(s)}$ by 
\[\cP_l^{(s)}=\sum_{i\in I}\delta_{(\zeta_{u_i},\omega^i)}.\] Similarly, one can define the point measure $\cP_r^{(s)}$, by replacing $[0,s]$ by $[s,\sigma]$. 
The following proposition (which a consequence of \cite[Proposition 3.5]{Randomtrees} and \cite[Lemma 3.7]{Randomtrees}) makes the link between the trees $\cT_a$ and the labelled tree associated to the measure $\N_0$. 
\begin{proposition}\label{decomposition}
    Let $M_p(\R_+\times\S)$ be the set of point measures on $\R_+\times\S$. Then, for any non-negative Borel measurable function $F$ on $\cW\times M_p(\R_+\times\S)^2$, 
    \[\N_0\left(\int_0^\sigma F(W_s,\cP_l^{(s)},\cP_r^{(s)})ds\right)=\int_0^\infty\E\left[F(B^{(a)},\cN_l^{(a)},\cN_r^{(a)})\right]da.\]
\end{proposition}
\begin{remark}
    This decomposition is very similar to the one of \cite[Proposition 2]{Spinedecomposition}.
\end{remark}
We can finally prove that we can restrict our study to the non-existence of $3^+$-hubs in $\cS_1$.
\begin{proposition}\label{equivalence}
    There exists a $3^+$-hub in the standard Brownian sphere with a positive probability if and only if with a positive probability, $(x_0,x_1,x_*)$ borders a $3^+$-hub in $\cS_1$. 
\end{proposition}
\begin{proof}
    First, by scaling, we can replace the standard Brownian sphere by a free Brownian sphere under $\N_0$. Then, by \cite[Proposition 3]{Geodesicstars}, for every non-negative measurable function $F$ on the space of three-pointed measure metric spaces, we have 
\[\N_0\left(\frac{1}{\sigma}\frac{1}{\sigma}\int\!\int\!\int F(\cS,x,y,z)\mu(dx)\mu(dy)\mu(dz)\right)=\N_0\left(\int F(\cS,x_*,x_0,y)\mu(dy)\right).\]
By Proposition \ref{decomposition}, this gives 
\begin{equation}\label{resample}    
\N_0\left(\frac{1}{\sigma}\frac{1}{\sigma}\int\!\int\!\int F(\cS,x,y,z)\mu(dx)\mu(dy)\mu(dz)\right)=\int_0^\infty\E[F(\cS_a,x_*,x_0,x_a)]da.
\end{equation}
By Lemma \ref{mesurable}, we can take $F(E,x,y,z)=\1_{\{(x,y,z)\text{ borders a $3^+$-hub in $E$}\}}$. For this choice of $F$ and by Proposition \ref{reduction}, we see that the left-hand side of \eqref{resample} is positive (in fact, infinite) if and only if there exists a $3^+$-hub in a standard Brownian sphere with positive probability. Furthermore, by scaling arguments, the right-hand side of \eqref{resample} is positive if and only if there is a positive probability that $(x_0,x_1,x_*)$ border a $3^+$-hub in $\cS_1$, which concludes the proof. 
\end{proof}

\subsection{Identifying the coalescence point}

 By Proposition \ref{equivalence}, we want to prove that almost surely, $(x_0,x_1,x_*)$ does not border a $3^+$-hub. In this section, we study the geodesic network bordered by $(x_0,x_1,x_*)$ in $\cS_1$. More precisely, we identify where the geodesics from $x_0$ and $x_1$ to $x_*$ merge, and show how to study these geodesics near their merging point.

To lighten notations, we set $(B^{(1)},\mathcal{N}_l^{(1)},\mathcal{N}_r^{(1)})=(B,\mathcal{N}_l,\mathcal{N}_r)$, and write $I,J$ for sets indexing the atoms of $\mathcal{N}_l^{(1)}$ and $\mathcal{N}_r^{(1)}$. We still denote by $u_*$ the element of $\cT_1$ with minimal label. As mentioned previously, the random surface $\cS_1$ comes with three distinguished points $(x_0,x_1,x_*)$, which are
\[x_0=p_{\cS_1}(\rho_0),\quad x_1=p_{\cS_1}(\rho_1),\quad x_*=p_{\cS_1}(u_*).\]
\begin{figure}
        \centering
        \includegraphics[scale=0.4]{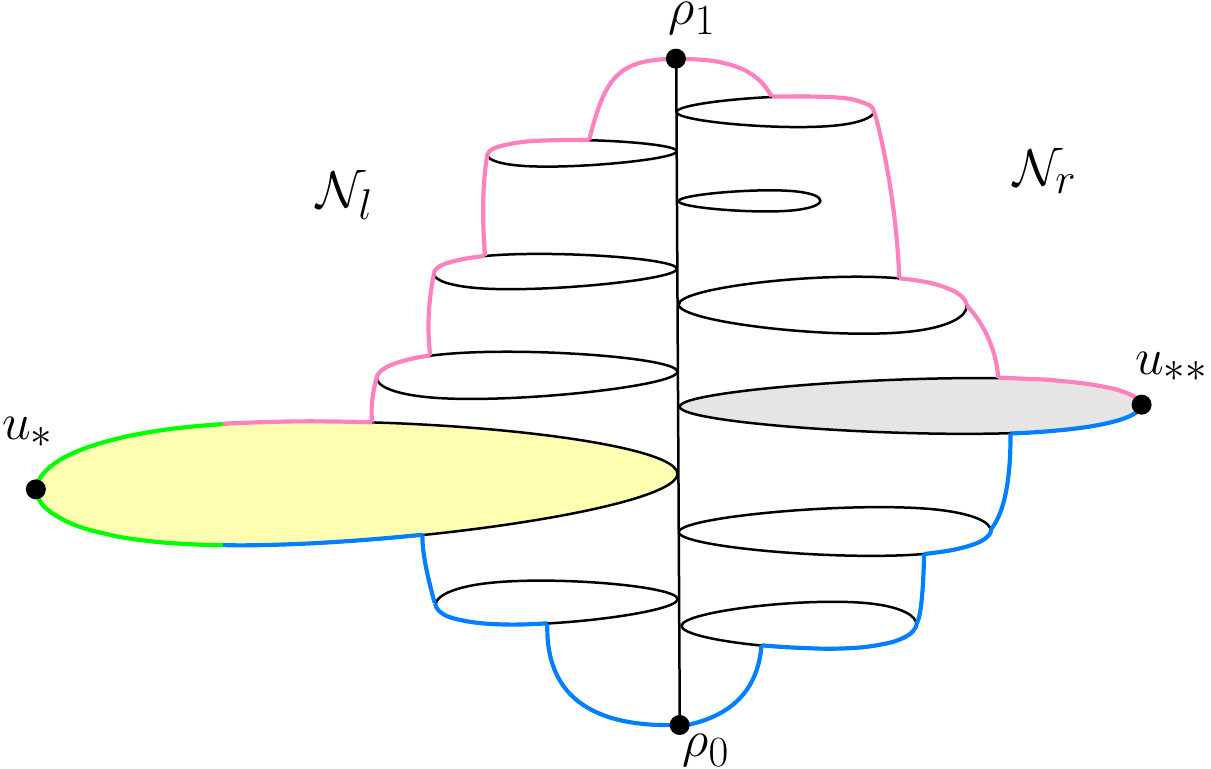}
        \caption{Representation of $T_0,\,\widehat{T}_0,\,T_1$ and $\widehat{T}_1$ in the random tree $\cT_1$. Note that the two blue (respectively pink, green) portions correspond to the same path in $\cS_1$.}
        \label{Arbre}
\end{figure}
In what follows, we will abuse notations by considering infimum and supremum in intervals of $\cT_1$. However, as explained in Section \ref{Triple}, since $\cT_1$ can be represented by a snake trajectory, these infimums and supremums are in fact infimums and supremums in some interval $[0,\sigma]$. For every $s\in[0,-\ell_*]$, we define 
\begin{align*}
    T_0(s)=\inf\{u\in[\rho_0,u_*],\ell_u=-s\}\quad\text{ and }\widehat{T}_0(s)=\sup\{u\in[u_*,\rho_0],\ell_u=-s\}.
\end{align*}
Similarly, for every $s\in[-\ell_{\rho_1},-\ell_*]$, we set
\begin{align*}
    T_1(s)=\inf\{u\in[\rho_1,u_*],\ell_u=-s\}\quad\text{ and }\widehat{T}_1(s)=\sup\{u\in[u_*,\rho_1],\ell_u=-s\}.
\end{align*}

Then, we define
\begin{align*}    \gamma_0(t)&=p_\cS(T_0(t))=p_\cS(\widehat{T}_0(t))\quad\text{ for }0\leq t\leq-\ell_*,\\    \gamma_1(t)&=p_\cS(T_1(t))=p_\cS(\widehat{T}_1(t))\quad\text{ for }-\ell_{\rho_1}\leq t\leq-\ell_*.
\end{align*}
Using \eqref{Bound} and the inequality $D\leq D^\circ$, it is easy to see that $\gamma_0$ (respectively $\gamma_1$) is a geodesic from $x_0$ to $x_*$ (respectively $x_1$ to $x_*$).
Moreover, by the main result of \cite{geodesic1}, the geodesics $\gamma_0$ and $\gamma_1$ are almost surely the unique such geodesics (the result is stated for the Brownian sphere $\cS$, but it also holds for $\cS_1$ by \eqref{resample} and scaling).

Without loss of generality, suppose that $u_*\in[\rho_0,\rho_1]$. Define $u_{**}$ as the unique element of $[\rho_1,\rho_0]$ such that $\ell_{u_{**}}=\inf\{\ell_u,u\in[\rho_1,\rho_0]\}$. In particular, let $i_*$ and $j_*$ be the indices of the unique atoms of $\cN_l$ and $\cN_r$ that contains the elements of minimal label on each side. Let us denote these elements by $u_{i_*}$ and $u_{j_*}$. 
Then, $u_*$ is the element such that \[\ell_{u_*}=\ell_{u_{i_*}}\wedge\ell_{u_{j_*}},\]
and $u_{**}$ is the one that satisfies 
\[\ell_{u_{**}}=\ell_{u_{i_*}}\vee\ell_{u_{j_*}}\]
Also, set $\ell_{**}=\ell_{u_{**}}$ and $x_{**}=p_{\cS_1}(u_{**})$. 
Observe that for every $t\in[0,-\ell_{**}),\,s\in[-\ell_{\rho_1},-\ell_{**})$,
\[D^\circ(\widehat{T}_0(t),T_1(s))>0,\]
whereas, for $t\in [-\ell_{**},-\ell_*]$,
\[D^\circ(\widehat{T}_0(t),T_1(t)=0.\]
Therefore, by Lemma \ref{Identification}, for every $t\in[0,-\ell_{**}),\,s\in[0,-\ell_{**}+\ell_{\rho_1})$,
\[\gamma_0(t)\neq\gamma_1(s))\] whereas for every $t\in [-\ell_{**},-\ell_*]$, 
\[\gamma_0(t)=\gamma_1(t+\ell_{\rho_1}).\]
In particular, the geodesics $\gamma_0$ and $\gamma_1$ coalesce at $x_{**}$. 
\begin{figure}
        \centering
        \includegraphics[scale=0.5]{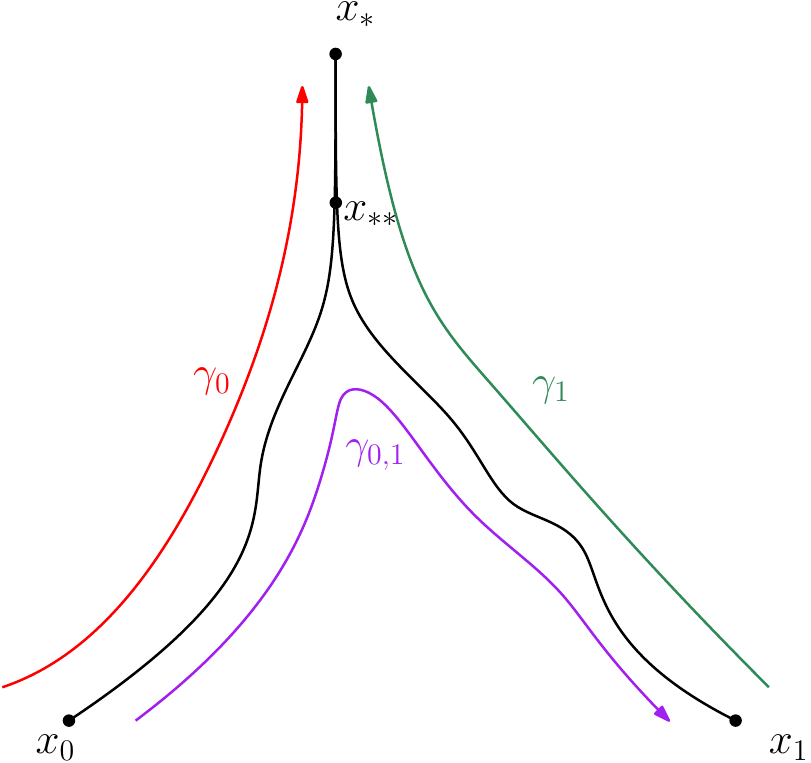}
        \caption{Illustration of the path $\gamma_{0,1}$. We need to determine whether it can be a geodesic.}
        \label{new path}
    \end{figure}

Let $\gamma_{0,1}$ be the path obtained by following $\gamma_0$ from $x_0$ to $x_{**}$, and then $\gamma_1$ from $x_{**}$ to $x_1$ (see Figure \ref{new path}). Since there is almost surely a unique geodesic between $x_0$ and $x_1$, the triple $(x_*,x_0,x_1)$ borders a $3^+$-hub if and only if $\gamma_{0,1}$ is a geodesic.

As mentioned previously, we need to study the geodesics $\gamma_0$ and $\gamma_1$ ``near $x_{**}$''. To do so, let $W_{\min}$ be the atom that contains $u_{**}$. The following proposition characterizes the law of $W_{\min}$.
\begin{proposition}\label{law}
   The law of $W_{\min}$ given $B$ is absolutely continuous with respect to $\int_0^1\N_{B_s}(dW)$.
\end{proposition}
\begin{proof}
   Recall that $i_*$ and $j_*$ be the indices of the unique atoms of $\mathcal{N}_l$ and $\mathcal{N}_r$ that contain the minimum on each side. For every $w\in\mathfrak{W}_0$, let $\mathcal{M}(dt,d\omega)$ and $\mathcal{M'}(dt,d\omega)$ be two independent Poisson point measures on $\R_+\times C(\R_+,\mathfrak{W})$, defined on some probability space with probability measure $\Pi_w$, with intensity :
\begin{equation*}
 2\mathbf{1}_{[0,\zeta_{(w)}]}(t)dt\N_{w(t)}(d\omega).   
\end{equation*}
Using Palm's formula and \eqref{inf}, we obtain : 
    \begin{align*}      &\E\left[F(W_{i_*},W_{j_*})\,|\,B\right]=\E\left[\sum_{i\in I,j\in J}F(W_i,W_j)\1_{i=i_*,j=j_*}\,\bigg|\,B\right]\\
        &=\int_0^1 \int_0^1 dsds'\int_{\mathfrak{S}\times \mathfrak{S}}\N_{B_s}(dW_1)\N_{B_{s'}}(dW_2) F(W_1,W_2)\\
        &\quad\quad\quad\quad\times\Pi_B\bigg[\cM\Bigl((t_i,\omega_i):(\omega_i)_*<(W_1)_*\Bigl)=0,\,\cM'\left((t'_j,\omega'_j):(\omega'_j)_*<(W_2)_*\right)=0\bigg]\\
         &=\int_0^1\int_0^1 dsds'\int_{\mathfrak{S}\times \mathfrak{S}}\N_{B_s}(dW_1)\N_{B_{s'}}(dW_2) F(W_1,W_2)\exp\left(-3\int_0^1\left(\frac{1}{(B_u-(W_1)_*)^2}+\frac{1}{(B_u-(W_2)_*)^2}\right)du\right).
    \end{align*}
    This gives us the joint law of $(W_{i_*},W_{j_*})$. In particular, we have 
    \begin{align*}
        [\E[F(W_{\min})\,|\,B]=2\int_0^1\int_0^1 &dsds'\int_{\mathfrak{S}\times \mathfrak{S}}\N_{B_s}(dW_1)\N_{B_{s'}}(dW_2) F(W_1)\1_{(W_1)_*<(W_2)_*}\\
        &\times\exp\left(-3\int_0^1\left(\frac{1}{(B_u-(W_1)_*)^2}+\frac{1}{(B_u-(W_2)_*)^2}\right)du\right),
    \end{align*}
    
    which gives the result.
\end{proof}
\begin{figure}
        \centering
        \includegraphics[scale=0.35]{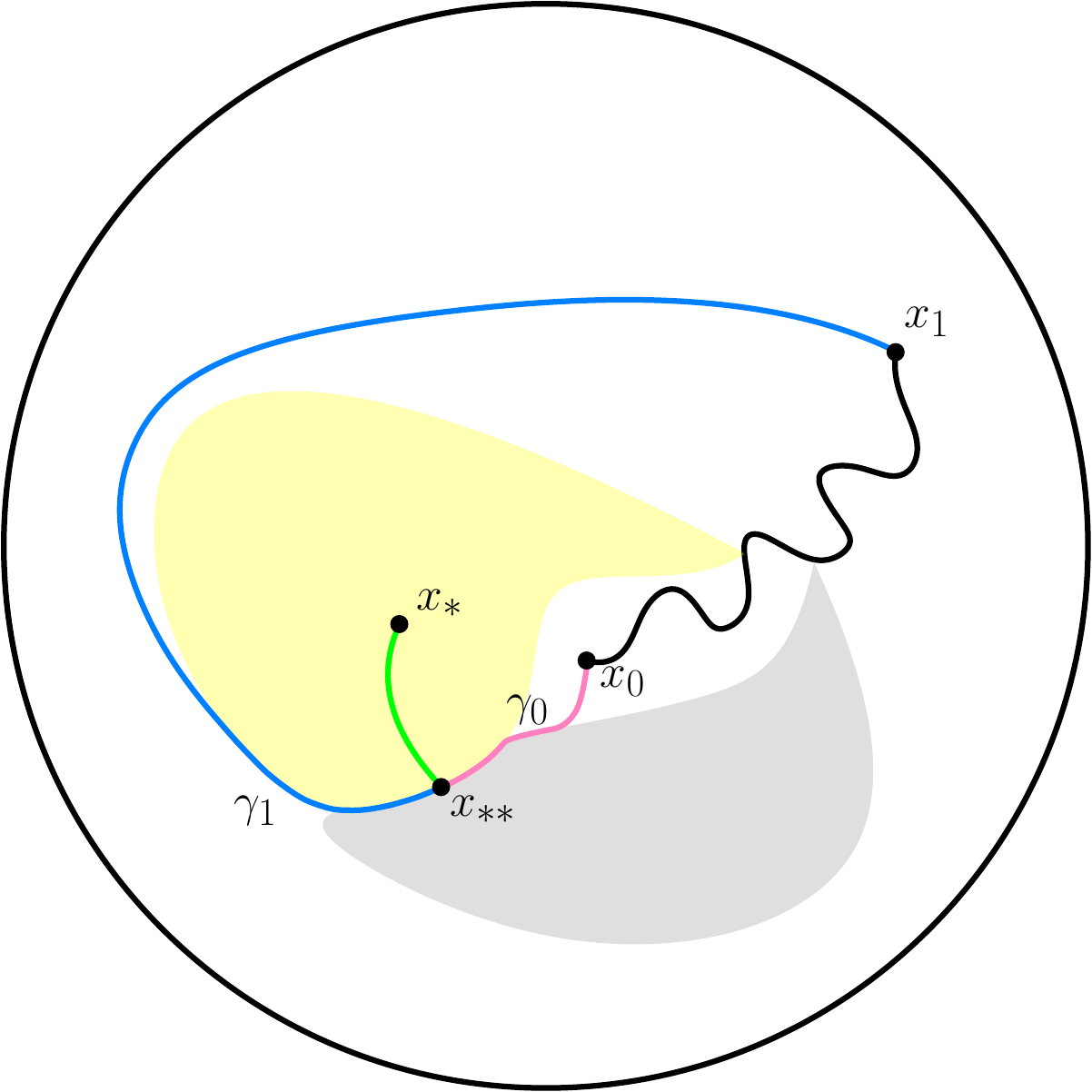}
        \caption{The same illustration as Figure \ref{Arbre}, but viewed in $\cS_1$. The black path corresponds to the projection of the spine of $\cT_1$, and the yellow (respectively gray) area is the projection of the atom containing $u_*$ (respectively $u_{**}$).}
\end{figure}
Let $\cT_{\min}$ be the labelled subtree associated to $W_{\min}$ The following lemma guarantees that $p_{\cS_1}(\cT_{\min})$ contains a non-trivial portion of $\gamma_0$ and $\gamma_1$. 
\begin{lemme}
Almost surely, 
\[\gamma_0\cap p_{\cS_1}(\cT_{\min})\neq\{x_{**}\}\quad\text{ and }\gamma_1\cap p_{\cS_1}(\cT_{\min})\neq\{x_{**}\}\]
\end{lemme}
\begin{proof}
    Without loss of generality, we can suppose that $x_{**}\in p_{\cS_1}(\cN_r)$. Set $\underline{B}=\inf_{s\in[0,1]}B_s$. Note that conditionally on $B$, for every $\varepsilon>0$, the number of atoms $W_i$ of $\cN_r$ such that $(W_i)_*<\underline{B}-\varepsilon$ follows a Poisson distribution with parameter 
    \[2\int_0^1\N_{B_s}(W_*<\underline{B}-\varepsilon)ds=3\int_0^1\frac{1}{(B_s-\underline{B}+\varepsilon)^2}ds<+\infty.\]
    In particular, this implies that there is no accumulation point in the set $\{(W_i,t_i),(W_i)_*<\underline{B}\}$. Since $W_{\min}$ belongs to this set, this means that almost surely, there exists $\delta>0$ such that for every atom $W_i\in\cN_r$ (which is different from $W_{\min}$), we have $(W_i)_*>(W_{\min})_*+\delta$. Therefore, for every $0<\varepsilon<\delta$, we have 
    \[\widehat{T}_0(-\ell_{**}-\varepsilon)\in \cT_{\min}\quad\text{ and }\quad  T_1(-\ell_{**}-\varepsilon)\in \cT_{\min},\] 
    which concludes the proof.
\end{proof}

Let $\tilde{\cS}$ be the random slice associated to $W_{\min}$.  In what follows, we will study the geodesics $\gamma_0$ and $\gamma_1$ ``restricted to $\tilde{\cS}$''. More precisely, if $0<a_0<b_0$ and $0<a_1<b_1$ are such that $\gamma_0|_{[b_0-a_0,b_0]}\subset p_\cS(\cT_{\min})$ and $\gamma_1|_{[b_1-a_1,b_1]}\subset p_\cS(\cT_{\min})$, with $p_{\cS_1}(\gamma_0(b_0))=p_{\cS_1}(\gamma_1(b_1))=x_{**}$, we set 
\[\tilde{\gamma}(t)=\left\{
    \begin{array}{ll}
   p_{\tilde{\cS}}(T_0(\ell_{**}+t-a_0)\quad&\text{ if $t\in[0,a_0]$} \\
     p_{\tilde{\cS}}(T_1(\ell_{**}-(t-a_0))\quad&\text{ if $t\in[a_0,a_0+a_1]$.}
    \end{array}\right. \]

Since $\tilde{\gamma}$ is a sub-path of $\gamma_{0,1}$, we already know that if $(x_*,x_0,x_1)$ borders a $3^+$-hub in $\cS_1$, then $\tilde{\gamma}$ is a geodesic in $\cS_1$. However, is not clear it is also a geodesic \textbf{in $\tilde{\cS}$}, which is the content of the following proposition.

\begin{proposition}\label{implication}
  Suppose that $(x_*,x_0,x_1)$ borders a $3^+$-hub in $\cS_1$. Then, $\tilde{\gamma}$ is a geodesic in $\tilde{\cS}$.  
\end{proposition}
\begin{proof}
   First, note that $\tilde{\gamma}|_{[0,a_0]}$ (respectively $\tilde{\gamma}|_{[a_0,a_0+a_1]}$) is just a portion of the geodesic $\tilde{\gamma}^{(r)}$ (respectively $\tilde{\gamma}^{(l)}$). Therefore, we only need to compute the distance between $\tilde{\gamma}(s)$ and $\tilde{\gamma}(t)$ for every $s\in[0,a_0]$ and $t\in[a_0,a_0+a_1]$. Using the inequality $\tilde{D}\geq D$ and the fact that $\gamma_0$ and $\gamma_1$ border a $3^+$-hub, we have 
    \[\tilde{D}(\tilde{\gamma}(s),\tilde{\gamma}(t))\geq D(\gamma_0(s+(b_0-a_0)),\gamma_1(b_0-(t-a_0)))=t-s.\]
On the other hand, the bound $\tilde{D}\leq \tilde{D}^\circ$ gives 
\[\tilde{D}(\tilde{\gamma}(s),\tilde{\gamma}(t))\leq \tilde{D}^\circ(T_0(-\ell_{**}+s-a_0),T_1(-\ell_{**}-(t-a_0))=t-s,\]
which concludes the proof. 
\end{proof}

\subsection{Geodesics in the Brownian slice}

In this section, we study geodesics in the Brownian slice to prove Theorem \ref{hub}. By Proposition \ref{implication}, $x_{**}$ is a $3^+$-hub in the Brownian sphere if and only if $x_{**}$ is in the interior of a geodesic in the random slice associated to $W_{\min}$. Therefore, the rest of this paper is devoted to prove that this does not happen, almost surely. By scaling, it is enough to prove this result for a Brownian slice under $\N_1(\cdot\,|\,W_*=0)$, for which we have a construction based on a spine decomposition and well suited to our problem. We will prove the following result.

\begin{theorem}\label{passes through}
Consider a Brownian slice $\tilde{\cS}$ distributed under $\N_1(\cdot\,|\,W_*=0)$. Then, almost surely, there is no geodesic that passes through $x_*$.
\end{theorem}
Note that Theorem \ref{passes through} implies Theorem \ref{hub}. Indeed, by scaling and Proposition \ref{law}, Theorem \ref{passes through} implies that no geodesic passes through $x_{**}$ in the slice associated to $W_{\min}$.

To prove this theorem, we will use a spine decomposition of the labelled tree $\cT$ under $\N_1(\cdot\,|\,W_*=0)$, that we recall here. This construction is a consequence of results in \cite{bessel}, and was already used in \cite{brownianplane,Hullprocess2016,hausdorff}.

Consider a triple $(R,\cN_l,\cN_r)$ defined on a probability space $\left(\Omega,\cF,\P\right)$, where:
\begin{itemize}[label=\textbullet]
    \item $R=(R_t)_{t\in[0,\tau_0]}$ is a Bessel process of dimension -5, starting at 1 and stopped when it reaches $0$,
    \item Given $R$, $\cN_l$ and $\cN_r$ are two independent Poisson point measure, with intensity 
    \[2\1_{[0,\tau_0]}\1_{\omega_*>0}\N_{R_t}(d\omega)dt.\]
\end{itemize}
Then, the random labelled tree $\cT$ associated to this triple, as explained in Section \ref{Triple}, is distributed as the random tree encoded by $\N_1(\cdot\,|\,W_*=0)$. Consequently, the random slice $\tilde{\cS}$ encoded by $(R,\cN_l,\cN_r)$ (as explained in Section \ref{slice}) is distributed as a Brownian slice under $\N_1(\cdot\,|\,W_*=0)$. 

For every $0\leq\beta\leq1$, set 
\[\tau_\beta=\inf\{t\geq0:R_t=\beta\}.\]
Note that the spine of $\cT$, which can be identified with the interval $[\tau_1,\tau_0]$, has a random length. The benefit of this construction is that $p_{\tilde{\cS}}(\tau_0)=x_*$. In particular, the point $u_*$ of the tree $\cT$ is exactly the top of the spine, which is identified with $\tau_0$. We also recall a particular case of Nagazawa's time reversal theorem (see \cite[Theorem VII 4.5]{revuzyor} and \cite[Exercise XI 1.23]{revuzyor}). Let $X=(X_t)_{t\in[0,S_1]}$ be a Bessel process of dimension $9$, starting from $0$ and stopped at its last hitting time of $1$, denoted by $S_1$. Then, the processes 
\begin{equation}\label{Time reversal}
    (R_{\tau_0-t})_{t\in[0,\tau_0]}\quad\text{and}\quad (X_t)_{t\in[0,S_1]}
\end{equation}
have the same law. 

We start by proving a weaker version of Theorem \ref{passes through}. 

\begin{proposition}\label{positive}
We have     
\[\P\left(\text{$x_*$ is not in the interior of a geodesic }\right)>0.\]
\end{proposition}
\begin{proof}
In what follows, we write $[\cdot\,,\,\cdot]_\cT$ to emphasize that we consider intervals on the tree $\cT$. For any $x,y\in\tilde{\cS}$ and $u,v\in\cT$ such that $p_{\tilde{\cS}}(u)=x$ and $p_{\tilde{\cS}}(v)=y$, we have
\[\tilde{D}(x,x_*)=\tilde{D}^\circ(u,u_*)=\ell_x,\quad \tilde{D}(y,x_*)=\tilde{D}^\circ(v,u_*)=\ell_y.\]
Therefore, if a geodesic between $x$ and $y$ passed through $x_*$, we would have 
\[\Tilde{D}(x,y)=\tilde{D}^\circ(u,v)=\ell_u+\ell_v.\]
Recall that $\tau_1$ identified with the bottom of the spine. First, note that if $u,v\in]\tau_1,\tau_0]_\cT$, we have
\[\tilde{D}^\circ(u,v)=\ell_u+\ell_v-2\max\left(\min_{w\in[u,v]_\cT}\ell_w,\min_{w\in[v,u]_\cT}\ell_w\right)>\ell_u+\ell_v.\]
This inequality still holds if $u,v\in]\tau_0,\tau_1]_\cT$.
Therefore, excluding these two cases, we just need to prove the result for $\tau_0\in[u,v]_\cT$. Then, observe that for topological reasons, the geodesic $\gamma$ between $u$ and $v$ must cross the curve $p_{\tilde{\cS}}((\tau_t)_{0\leq t\leq1})$ (which is the projection of the spine). We will prove that $x_*$ is never the best point to cross this curve.

Fix $0<\beta<1$, and suppose that $u\in[\tau_1,\tau_\beta]_\cT$ and $v\in[\tau_\beta,\tau_1]_\cT$. Then, 
\begin{equation}\label{majoration}
    \tilde{D}(u,v)\leq \tilde{D}^\circ(u,\tau_\beta)+\tilde{D}^\circ(\tau_\beta,v)=\ell_u+\beta-2\min_{w\in[u,\tau_\beta]_\cT}\ell_w+\ell_v+\beta-2\min_{w\in[\tau_\beta,v]_\cT}\ell_w.
\end{equation}

In particular, if there exists $0<\beta<1$ such that 
\begin{equation}\label{Minimum control}
2\min_{w\in[\tau_1,\tau_\beta]_\cT}\ell_w\wedge2\min_{w\in[\tau_\beta,\tau_1]_\cT}\ell_w>\beta
\end{equation}
then by \eqref{majoration}, 
\[\tilde{D}^\circ(u,v)<\ell_u+\ell_v\]
which would imply that the geodesic does not pass through $x_*$. Therefore, we need to show that the inequality \eqref{Minimum control} holds for some $\beta$ arbitrary small, almost surely. To this end, for every $n\in\N$, we introduce the event 
\[E_n=\left\{\min_{w\in[\tau_1,\tau_{2^{-n}}]_\cT\cup[\tau_{2^{-n}},\tau_1]_\cT}\ell_w>2^{-n-1}\right\}.\] The previous discussion implies that
\begin{equation}\label{inclusion}
    \{\limsup E_n\}\subset\{\text{No geodesic passes through }x_*\}.
\end{equation}
Using properties of Poisson point measures, formula \eqref{inf} and the time reversal property \eqref{Time reversal}, we have 
\begin{align*}
    \P(E_n)&=\E\left[\exp\left(-4\int_{\tau_1}^{\tau_{2^{-n}}}\N_{R_t}(0<W_*<2^{-n-1})dt\right)\right]\\
    &=\E\left[\exp\left(-6\int^{\tau_{2^{-n}}}_{\tau_1}\frac{1}{(R_t-2^{-n-1})^2}-\frac{1}{(R_t)^2}dt\right)\right]\\
     &=\E\left[\exp\left(-6\int_{S_{2^{-n}}}^{S_1}\frac{1}{(X_t-2^{-n-1})^2}-\frac{1}{(X_t)^2}dt\right)\right]\\
    &=\E\left[\exp\left(-6\int_{S_1}^{S_{2^n}}\frac{1}{(X_t-1/2)^2}-\frac{1}{(X_t)^2}dt\right)\right]
\end{align*} where $S_t$ stands for the last hitting time of $t$ by $X$
(the last equality follows from the scaling properties of Bessel processes).
Note that the integral is finite for every $n\geq1$, due to the well-known fact that for every $\varepsilon\in(0,1/6)$ and $t$ large enough, $X_t>t^{1/2-\varepsilon}$ a.s. Moreover, for every $n\leq m$, $\P(E_n)\geq\P(E_m)$. In this situation, we cannot use the Borel-Cantelli lemma to conclude because the events $(E_n)_{n\geq1}$ are not independent. However, for $n>m$, we have
\[E_n\cap E_m=\left\{\inf_{w\in[\tau_1,\tau_{2^{-m}}]_\cT\cup[\tau_{2^{-m}},\tau_1]_\cT}\ell_w>2^{-m-1}\right\}\cap\left\{\inf_{w\in[\tau_{2^{-m}},\tau_{2^{-n}}]_\cT\cup[\tau_{2^{-n}},\tau_{2^{-m}}]_\cT}\ell_w>2^{-n-1}\right\}.\] 
Using the fact that the processes $(R_t)_{\tau_1\leq t\leq \tau_{2^{-m}}}$ and $(R_t)_{\tau_{2^{-m}}\leq t\leq \tau_{2^{-n}}}$ are independent, we obtain 
\begin{align*}
    \P(E_n\cap E_m)&=\E\left[\exp\left(-6\int_{\tau_1}^{\tau_{2^{-m}}}\frac{1}{(R_t-2^{-m-1})^2}-\frac{1}{(R_t)^2}dt\right)\exp\left(-6\int_{\tau_{2^{-m}}}^{\tau_{2^{-n}}}\frac{1}{(R_t-2^{-n-1})^2}-\frac{1}{(R_t)^2}dt\right)\right]\\
    &=\E\left[\exp\left(-6\int_{\tau_1}^{\tau_{2^{-m}}}\frac{1}{(R_t-2^{-m-1})^2}-\frac{1}{(R_t)^2}dt\right)\right]\E\left[\exp\left(-6\int_{\tau_{2^{-m}}}^{\tau_{2^{-n}}}\frac{1}{(R_t-2^{-n-1})^2}-\frac{1}{(R_t)^2}dt\right)\right]\\
    &=\P(E_m)\E\left[\exp\left(-6\int_{S_1}^{S_{2^{n-m}}}\frac{1}{(X_t-1/2)^2}-\frac{1}{(X_t)^2}dt\right)\right]\\   
    &=\P(E_m)\P(E_{n-m})
\end{align*}
which we rewrite as $C_{n,m}\P(E_n)\P(E_m)$, where
\[C_{n,m}=\frac{\P(E_{n-m})}{\P(E_n)}. \]
Define 
\begin{equation*}   p_\infty=\lim_{k\rightarrow\infty}\P(E_k)=\E\left[\exp\left(-6\int_{S_1}^{\infty}\frac{1}{(X_t-1/2)^2}-\frac{1}{(X_t)^2}dt\right)\right].
\end{equation*}
Since the integral is finite, we have $p_\infty>0$, and
\[C_{n,m}=\frac{\P(E_{n-m})}{\P(E_n)}\leq\frac{1}{p_\infty}<\infty.\]
Therefore, by the Kochen-Stone lemma (see \cite{KochenStone} and \cite{KochenStone2}),
\[\P(\limsup E_n)>0.\]
This gives the result, using \eqref{inclusion}.
\end{proof}

To conclude the proof of Theorem \ref{passes through}, we rely on a $0-1$ law argument. By \eqref{Time reversal}, the random slice $\tilde{\cS}$ can be constructed from the triple $(X,\widehat{\cN}_l,\widehat{\cN}_r)$, where
\begin{itemize}[label=\textbullet]
    \item $X=(X_t)_{t\in[0,S_1]}$ is a Bessel process of dimension 9, starting at 0 and stopped when it reaches $1$ for the last time,
    \item Given $X$, $\widehat{\cN}_l$ and $\widehat{\cN}_r$ are two independent Poisson point measure, with intensity 
    \[2\1_{[0,S_1]}\1_{\omega_*>0}\N_{X_t}(d\omega)dt.\]
\end{itemize}
For the rest of this paper, we will work with this construction. Let $I$ and $J$ be sets indexing the atoms of the Poisson point measures $\widehat{\cN}_l$ and $\widehat{\cN}_r$, so that
\[ \widehat{\cN}_l=\sum_{i\in I}\delta_{(t_i,\cT_i)} \quad\text{and}\quad \widehat{\cN}_r=\sum_{j\in J}\delta_{(t_j,\cT_j)}. \]
Let $T_1=\inf\{t\geq 0,X_t=1\}.$
Then, set $\cF_n=\sigma((X_t)_{0\leq t\leq 1/n\wedge T_1},(\cT_i)_{0\leq t_i\leq 1/n\wedge T_1},(\cT_j)_{0\leq t_j\leq 1/n\wedge T_1})$. 

\begin{lemme}\label{0-1}
    The $\sigma$-algebra 
    \[\bigcap_{n\geq1}\cF_n\]
    is trivial.
\end{lemme}
\begin{proof}
    Let $\cN_l^*$ and $\cN_r^*$ be two independent Poisson point measures on $\R_+\times\cS$, with intensities 
     \[2\N_{0}(d\omega)dt,\]
     and let $I^*$ and $J^*$ be sets indexing the atoms of these measures.
     As $(X,\widehat{\cN}_l,\widehat{\cN}_r)$ can be expressed as a measurable function of $(X,\cN^*_l,\cN_r^*)$, it is enough to prove the result with $\cF_n^*$ instead of $\cF_n$, where 
     \[\cF_n^*=\sigma((X_t)_{0\leq t\leq 1/n\wedge T_1},(\cT_i)_{i\in I^*,0\leq t_i\leq 1/n},(\cT_j)_{j\in J^*,0\leq t_j\leq 1/n}).\]
     We will use the following result (see \cite[Exercise II.2.15]{revuzyor}):
     consider some $\sigma$-algebra $\mathcal{H}$ and $\mathcal{G}_0\subseteq\mathcal{G}_1\subseteq...$ such that $\mathcal{H}$ and $\mathcal{G}_0$ are independent. Then, 
     \begin{equation}\label{Intersection tribu}
     \sigma(\mathcal{H},(\mathcal{G}_i)_{i\in\N})=\bigcap_{i\in\N}\sigma(\mathcal{H},\cG_i).
     \end{equation}
     Set 
     \begin{align*}
         \cH_n&=\sigma((X_t)_{0\leq t\leq 1/n\wedge T_1}),\\
         \cG_n&=\sigma((\cT_i)_{i\in I^*,\,0\leq t_i<1/n}),\\
         \cG'_n&=\sigma((\cT_j)_{j\in J^*,\,0\leq t_i<1/n}), 
     \end{align*} and define 
     \[\cH=\bigcap_{n\geq1}\cH_n,\quad\cG=\bigcap_{n\geq1}\cG_n,\quad\cG'=\bigcap_{n\geq1}\cG'_n.\]
     By construction, for every $n,m,k\geq1$, $\cH_n,\cG_m$ and $\cG'_k$ are independent. First, we can apply \eqref{Intersection tribu} twice with $\cG$ and $(\cG'_n)_{n\geq1}$, which gives 
     \[\sigma(\cG,\cG')=\bigcap_{n\geq1}\sigma(\cG,\cG'_n)=\bigcap_{n\geq1}\bigcap_{m\geq1}\sigma(\cG_m,\cG'_n).\]
     Then, we can apply \eqref{Intersection tribu} a couple more times to $\sigma(\cG,\cG')$ and $(\cH_n)_{n\geq1}$, and we obtain 
     \[\sigma(\cH,\cG,\cG')=\bigcap_{n\geq1}\sigma(\cH_n,\sigma(\cG,\cG'))=\bigcap_{n,m,h\geq1}\sigma(\cH_n,\cG_m,\cG'_h).\]
     However, since our sequences of $\sigma$-algebras are decreasing, one can easily check that 
     \[\bigcap_{n,m,h\geq1}\sigma(\cH_n,\cG_m,\cG'_h)=\bigcap_{n\geq1}\sigma(\cH_n,\cG_n,\cG'_n),\]
     which gives 
     \[\sigma(\cH,\cG,\cG')=\bigcap_{n\geq1}\cF_n^*.\]
     
However, since a Bessel process of dimension $9$ is the norm of a Brownian motion in dimension $9$, $\cH$ is trivial by Blumenthal $0-1$ law. Similarly, independence properties of Poisson point measures imply that $\cG$ and $\cG'$ are also trivial. Therefore, $\sigma(\cH,\cG,\cG')$ is also trivial, which completes the proof.   
\end{proof}
We claim that for every $n\geq1$, the event 
\begin{equation}
    \{\text{There exists a geodesic that passes through }x_*\}
\end{equation}
belongs to $\cF_n$. To see this, set 
\[\theta_n=[0,X_{1/n\wedge T_1}]\cup\bigcup_{i\in I\cup J,0\leq t_i<1/n\wedge T_1}\cT_i.\]
This set is a random subtree of $\cT$, and is measurable with respect to $\cF_n$. We will need the following lemma.
\begin{lemme}\label{voisinage}
    Almost surely, for every $n\geq1$, $p_{\tilde{\cS}}(\theta_n)$ contains a neighborhood of $x_*$ in $\tilde{\cS}$.
\end{lemme}
\begin{proof}
    We argue by contradiction. If the statement did not hold, we could find $n_0\in\N$ and a sequence $(u_n)_{n\in\N}\in\cT$ such that for every $n\in\N$, 
    \[D(x_*,p_{\tilde{\cS}}(u_n))\leq1/n\quad\text{ and }\quad u_n\notin\theta_{n_0}.\]
    By compactness, we can suppose that the sequence $(u_n)$ converges toward some element $u\in Cl(\cT\backslash\theta_{n_0})$ such that $D(x_*,p_{\tilde{\cS}}(u))=0$. However, by Lemma \ref{Identification}, this is not possible, which concludes. 
\end{proof}

We can finally prove the main result of this section. 
\begin{proof}[Proof of Theorem \ref{passes through}.]
    By Lemma \ref{voisinage}, the event 
    \begin{equation}\label{event}
    \{\text{There exists a geodesic that passes through }x_*\}
\end{equation} belongs to $\bigcap_{n\geq1}\cF_n$. 
However, by Lemma \ref{0-1} this $\sigma$-algebra is trivial. Moreover, by Proposition \ref{positive}, the probability of the event \eqref{event} is strictly less than $1$. Therefore, it has probability $0$, which concludes the proof. 
\end{proof}

\section{Proof of Proposition \ref{alignement}}

In this section, we use Theorem \ref{hub} to prove Proposition \ref{alignement}.
 \begin{proof}[Proof of Proposition \ref{alignement}]
  We argue by contradiction, by proving that if the statement did not hold, there would be a $3^+$-hub in the Brownian sphere with positive probability. 
First, we know that almost surely, $x_*$ is not in the interior of a geodesic (see \cite[Corollary 7.7]{geodesic1}). Therefore, if $(x_*,p_{\cS}(u),p_{\cS}(v))$ are aligned, then either $p_{\cS}(u)$ is on a simple geodesic from $p_{\cS}(v)$, or $p_{\cS}(v)$ is on a simple geodesic from $p_{\cS}(u)$.

       Consider $u,v\in\cT$ such that $D^\circ(u,v)=D(u,v)$, and suppose that $(x_*,p_{\cS}(u),p_{\cS}(v))$ are not aligned. Without loss of generality, suppose that $u_*\notin[u,v]$, and as previously, define 
       \[u_{**}=\inf\left\{w\in[u,v]\,:\ell_w=\inf_{z\in[u,v]}\ell_z\right \}.\]
      Note that $u_{**}$ is different from $u$ and $v$, otherwise $(x_*,p_{\cS}(u),p_{\cS}(v))$ would be aligned. Set $\ell_{**}=\ell_{u_{**}}$. Define, for $t\in[0,D(u,v)]$,
      \[U(t)=\left\{
    \begin{array}{ll}
    \inf\left\{w\in[u,u_{**}]\,:\ell_w=\ell_u-t\right\} \quad&\text{ if $t\in[0,\ell_u-\ell_{**}]$}, \\
      \sup\left\{w\in[u_{**},v]\,:\ell_w=2\ell_{**}-\ell_u+t\right\} \quad&\text{ if $t\in[\ell_u-\ell_{**},\ell_u+\ell_v-2\ell_{**}]$.}
    \end{array}\right.\]
    and
    \[\gamma_{u,v}(t)=p_\cS(U(t)).\]
This path corresponds to following a simple geodesic from $p_\cS(u)$ up to $p_\cS(u_{**})$, and then another simple geodesic from $p_\cS(v)$, in reverse direction (this is very similar to the path $\gamma_{0,1}$, see Figure \ref{new path}). 
Let us show that $\gamma_{u,v}$ is a geodesic between $p_{\cS}(u)$ and $p_{\cS}(v)$. First, note that the restriction of $\gamma_{u,v}$ to $[0,\ell_u-\ell_{**}]$ (respectively $[\ell_u-\ell_{**},\ell_u+\ell_v-2\ell_{**}]$) is a portion of a simple geodesic. Therefore, we just need to show that for every $t\in[0,\ell_u-\ell_{**}],\,s\in[\ell_u-\ell_{**},\ell_u+\ell_v-2\ell_{**}]$, 
\begin{equation}\label{égalité}
    D(\gamma_{u,v}(t),\gamma_{u,v}(s))=s-t.
\end{equation}
By the triangle inequality, we have 
\begin{align*}
D^\circ(u,v)=D(u,v)&\leq D(p_{\cS}(u),\gamma_{u,v}(t))+D(\gamma_{u,v}(t),\gamma_{u,v}(s))+D(p_{\cS}(v),\gamma_{u,v}(s))\\
&=D^\circ(u,U(t))+D(\gamma_{u,v}(t),\gamma_{u,v}(s))+D^\circ(U(s),v).
\end{align*}
Since 
\[D^\circ(u,v)-D^\circ(u,U(t))-D^\circ(U(s),v)=\ell_{U(t)}+\ell_{U(s)}-2\ell_{**}=D^\circ(U(t),U(s)),\]
we have
\[D^\circ(U(t),U(s))\leq D(\gamma_{u,v}(t),\gamma_{u,v}(s)).\]
Since the converse inequality always holds, the previous line is an equality. Finally, note that 
\[D^\circ(U(t),U(s))=\ell_{U(t)}+\ell_{U(s)}-2\ell_{**}=\ell_u-t+2\ell_{**}-\ell_u+s-2\ell_{**}=s-t,\]
which gives \eqref{égalité}. However, this implies that $u_{**}$ is a $3^+$-hub, which is in contradiction with Theorem \ref{hub}. Therefore, $(x_*,p_{\cS}(u),p_{\cS}(v))$ are aligned, which concludes the proof. 
\end{proof}

\printbibliography
\end{document}